\documentclass{article}

\usepackage{fancyhdr}
\usepackage{extramarks}
\usepackage{amsmath}
\usepackage{amsthm}
\usepackage{amsfonts}
\usepackage{amssymb}
\usepackage{enumerate}
\usepackage{geometry}
\usepackage{hyperref}
\usepackage{cleveref}
\usepackage[maxnames = 1000, style=alphabetic, maxcitenames=2, url=false, isbn=false, doi=false, eprint=false, backend=biber]{biblatex}
\theoremstyle{definition}
\newtheorem{definition}{Definition}
\theoremstyle{definition}
\newtheorem{theorem}{Theorem}
\theoremstyle{definition}
\newtheorem{lemma}{Lemma}
\theoremstyle{definition}
\newtheorem{corollary}{Corollary}

\newtheoremstyle{named}{}{}{\itshape}{}{\bfseries}{.}{.5em}{\thmnote #3}
\theoremstyle{named}

\begin{document}
   \title{Mixing Times and Cutoff for the Rook's Walk} 
   \author{Jonatan Kaare-Rasmussen
   \thanks{University of California, Santa Barbara; 
   L\&S Mathematics and CCS Computing}}

   \date{\today}
   \maketitle

   \begin{abstract}
   We study the mixing time of the Rook's Walk Markov chain on a $d$-dimensional chess board of side length $n\geq 3$, where a rook moves by first selecting an axis uniformly at random and then selecting a new position along that axis uniformly from among the $n-1$ unoccupied alternatives. Our method is to lump the state space of the Rook's Walk by Hamming distance, yielding a birth-death Markov chain. We prove that this lumped birth-death chain has the same mixing time as the Rook's Walk and identify all eigenvalues and eigenfunctions of the projected chain. We then combine the eigenfunction lower bound approach of Wilson (2004) with an $L^2$ upper bound to obtain new sharpened bounds on the mixing time of the Rook's Walk. As a consequence, we show that the Rook's Walk Markov chain exhibits cutoff.
   \end{abstract}
   \section{Introduction}
   Given a $d$ dimensional chess board of length $n$ with a single rook, make a move uniformly at random from the set of valid rook moves. Recall that a rook can move any number of steps in any dimension. This defines the Markov chain called the Rook's Walk. In the language of Markov chains, this chain has state space $\Omega=\{1,\dots,n\}^d$ with transition probability from state $y\in \Omega$ to $x\in \Omega$ encoded in the transition matrix $P$ defined as,
   \[P(x,y) = \frac{1}{d(n-1)}\mathbf{1}_{\{||y-x||_0=1\}}\]
   where $||y-x||_0$ denotes the Hamming distance between $x$ and $y$ (i.e., the number of coordinates where $x$ and $y$ differ). This chain is ergodic, implying that it converges to a stationary distribution with respect to some notion of distance between distributions. It is standard to use the variation distance here, which for distributions $\mu$ and $\nu$ over $\Omega$ is defined as,
      \[||\mu-\nu||_\textsc{tv} = \sup_{A\subset \Omega} |\mu(A) - \nu(A)| = \frac{1}{2}\sum_{x\in \Omega}|\mu(x) - \nu(x)|\] 
   Here we study the mixing time of the Rook's Walk. Let $\pi$ be the stationary distribution and we can define the mixing time ($t_\text{mix}(\epsilon)$) in terms of the distance to stationary ($d(t)$) and time $t$, as follows.
      \[t_\text{mix}(\epsilon) = \inf \{t \mid d(t)\leq \epsilon\}, \quad \text{where} \quad d(t) =\max_{x\in \Omega}||P^t(x,\cdot) - \pi||_\textsc{tv}\]
   In \cite{steven_s._kim_mixing_2012} the Rook's Walk is decomposed into a product chain, which gives the eigenvalues of the transition matrix, resulting in the following upper and lower bounds obtained using spectral gap methods.
   \[\frac{d(n-1)}{n}\log\left(\frac{1}{2\epsilon}\right) \leq t_\text{mix}(\epsilon) \leq \frac{d(n-1)}{n}\log\left(\frac{n^d}{\epsilon}\right)\]
   In \cite{mcleman_mixing_2017} the upper bound is tightened using path coupling \cite{646111}. The upper bound obtained using this method is,
   \[t_\text{mix}(\epsilon) \leq \left\lceil\frac{\log\left(\frac{d}{\epsilon}\right)}{\log\left(\frac{d(n-1)}{(d-1)(n-1)+1}\right)}\right\rceil\]
   Here, we consider a projection or lumped chain of the Rook's Walk with the same mixing time as the standard Rook's Walk. The projection chain we consider is a birth-death chain, which has some special properties that allow us to find all eigenvalues \emph{and} all eigenfunctions of the projection chain. This allows us to make use of Wilson's lower bound \cite{Wilson_2004} and the $L^2$ upper bound \cite{diaconis_group_1988,diaconis_generating_1981,levin_markov_2017}
   to get the following tightened mixing time bounds.
      \begin{theorem}
         \label{main1}
      Let $t_\text{mix}(\epsilon)$ with $\epsilon >0$ be the mixing time of the Rook's Walk with dimension $d$ and board-length $n$. Then,
      \begin{multline*}
       \left(\log \left(\frac{d(n-1)}{2n}\right) + \log\left(\frac{1-\epsilon}{\epsilon}\right)\right)\left(2\log\left(\frac{d(n-1)}{d(n-1)-n}\right)\right)^{-1} \\
        \leq t_\text{mix}(\epsilon)  \leq-\frac{d(n-1)}{2n}\log\left((4\epsilon^2+1)^{1/d}-1\right) 
      \end{multline*}
   \end{theorem}
   \noindent Moreover, these bounds are sufficient to prove that the Rook's Walk exhibits the cutoff phenomenon. In general, the cutoff phenomenon is the property that as the state space grows, the chain transitions from unmixed to mixed abruptly. Specifically, for some sequence of Markov chains, indexed by $d\in \mathbb{N}$, we say that the sequence exhibits cutoff, if for any $\epsilon\in(0,1)$,
      \[\lim_{d\to \infty } \frac{t^{(d)}_\text{mix}(1-\epsilon)}{t^{(d)}_\text{mix}(\epsilon)}=1\]
   where $t_\text{mix}^{(d)}$ is the mixing time of the $d$th chain. In the case of the Rook's Walk, the sequence of Markov chains we consider is defined such that the $d$th chain is the Rook's Walk on the state space $\{1,\dots,n\}^d$ for some fixed $n$. Additionally, one can say something about the length of the period over which abrupt mixing occurs. Specifically we say a chain with parameter set $a$ has a cutoff at time $F(a)$ with window, $w(a)=o(F(a))$ if,
   \[F(a) + c_l(\epsilon) w(a) \leq t_\text{mix}(\epsilon) \leq F(a) + c_u(\epsilon) w(a)\]
   for sufficiently large parameter choices, where $c_l$ and $c_u$ are functions that depend only on $\epsilon$. For the Rook's Walk, we claim the following,
\begin{theorem}
         \label{main2}
   The Rook's Walk has cutoff at $\frac{d(n-1)}{2n}\log(d)$ with window $\frac{d(n-1)}{n}$. In other words, there exist functions $c_l(\epsilon), c_u(\epsilon)$ that only depend on $\epsilon$ such that,
   \[\frac{d(n-1)}{2n}\log(d) + c_l(\epsilon)\frac{d(n-1)}{n} \leq t_\text{mix}(\epsilon) \leq \frac{d(n-1)}{2n}\log(d) + c_u(\epsilon)\frac{d(n-1)}{n}\]
\end{theorem}
\noindent For context, the Rook's Walk is an example of a walk on a highly symmetric state space, and its natural lumping by Hamming distance is in many ways analogous to the classical reduction of the random walk on the hypercube to an Ehrenfest urn as treated in \cite{levin_markov_2017}. In these settings, the spectrum and eigenfunctions can often be written explicitly in an orthogonal polynomial basis, allowing the use of various powerful spectral techniques for mixing time analysis  \cite{saloff-coste_random_2004,levin_markov_2017}. Our contribution is to make this fully explicit for the Rook's Walk by (i) proving that our projection preserves total-variation mixing times, (ii) computing the complete spectral data for the projected birth-death chain, and (iii) combining an eigenfunction lower bound of Wilson with an $L^2$ upper bound to obtain sharp two-sided estimates and a cutoff statement with an explicit window.

\section{The Projection Chain}
In this section we define a projection chain of the Rook's Walk (Def. \ref{BDRW}) which projects each state in the state space of the Rook's Walk onto Hamming shells (Def. \ref{defphi}). We then show that the mixing time of this projection chain is identical to the mixing time of the full Rook's Walk (Corollary \ref{mixing_time_eq}).
\subsection{Mixing Time Equivalence}
Most of our analysis focuses on a projection of the standard Rook's Walk. With the correct projection, the mixing time is the same for both the projected chain and the standard Rook's Walk. This mixing time equivalence is only possible because the Rook's Walk is transitive, that is,
\begin{definition}
   \label{transitivity_def}
A Markov chain (with transition matrix $P$) is called {\bf transitive} if for each pair of states $(x,y)\in \Omega^2$, there is a bijection $\xi_{(x,y)}: \Omega\to \Omega$ such that,
\[\xi(x) = y\]
and,
\[P(z,w)=P(\xi(z),\xi(w)),\quad \forall z,w\in \Omega\]
\end{definition}
\begin{lemma}
   \label{transitivity}
  The Rook's Walk with dimension $d\in \mathbb{Z}_{\geq 2}$ and board length $n\in \mathbb{Z}_{\geq 3}$ is transitive.
  \begin{proof}
  Fix $x,y\in \Omega$ and let $d = (y-x \bmod n+1)$. Define, $\xi_{(x,y)}: \Omega \to \Omega$ as,
  \[v \overset{\xi_{(x,y)}}{\longmapsto} v + d,\quad  \quad \forall v\in \Omega\]
  Now, notice that,
  \begin{align*}
   \xi_{(x,y)}(x) &= x + d\\
   &= (x\bmod n+1) + (y-x \bmod n+1)\\
   &= (y\bmod n+1)\\
   &= y
  \end{align*}
  Therefore, $\xi_{(x,y)}$ satisfies our first condition. Now, fix $z,w \in \Omega$. Notice that the transition probability $P(z,w)$ is dependent only on the Hamming distance between $z$ and $w$ (Recall that $||z-w||_0$ is the Hamming distance). Further we have,
  \[||\xi_{(x,y)}(z) - \xi_{(x,y)}(w)||_0 = ||(z+d)-(w+d)||_0= ||z-w||_0\]
  That is, Hamming distance is invariant under $\xi_{(x,y)}$ and, by extension, \[P(z,w) = P(\xi_{(x,y)}(z), \xi_{(x,y)}(w))\]
  Therefore, the Rook's Walk is transitive.
  \end{proof}
\end{lemma}
\noindent The transitivity property allows us to consider the Rook's Walk from any arbitrary starting state. Therefore, we arbitrarily choose to start at the state $\hat{1}=\{1\}^d$. Now, our projection is onto the equivalence classes induced by the relation defined as the pairs of states in $\Omega$ that have the same Hamming distance from $\hat{1}$. This relation has equivalence classes that are called Hamming shells as defined below. 
\begin{definition}
   \label{defphi}
   Let $\Omega$ be the state space of a Rook's Walk with parameters $n$ and $d$ (i.e., $\Omega=\{1,\dots, n\}^d$). Define the $i^\text{th}$ Hamming shell $H_i$ to be,
   \[H_i = \{x\in \Omega \mid i= ||\hat{1}-x||_0\}, \quad \forall i\in \{0,\dots,d\}\]
   where $||\cdot||_0$ denotes Hamming distance and $\hat{1}:=\{1\}^d$. Furthermore, define $\varphi : \{0,\dots,d\} \to \mathbb{Z}_{\geq 0}$ such that 
   \[\varphi(i) = |H_i|, \quad  \forall i\in \{0,\dots,d\}\]
\end{definition}
\begin{lemma}
   Let $\varphi$ be defined as in Definition \ref{defphi}. Then,
   \[\varphi(i) = {d \choose i}(n-1)^i\]
   \begin{proof}
      Fix $i \in \{0,\dots,d\}$. Now, $\varphi(i)$ counts the number of elements in $\mathbb{Z}_n^d$ with Hamming distance of exactly $i$ from $\hat{1}$. That is, the number of elements with $i$ non-one elements. There are ${d \choose i}$ different choices for which dimensions should be non-one, and $(n-1)$ choices for each of these dimensions. 
   \[\therefore \quad \varphi(i) = {d \choose i}(n-1)^i,\quad \forall i \in \{0,\dots,d\}\]
   \end{proof}
\end{lemma}
\noindent Now, we finally define the projection chain.

\begin{definition}
   \label{BDRW}
  Define a {\bf Rook's Walk Birth-Death} (denoted RWBD) Markov chain over the state space $\Omega_* = \{0,\dots , d\}$ with a transition from $X_t \in \Omega_*$ to $X_{t+1}\in \Omega_*$ governed by parameter $n\in \mathbb{Z}_{\geq 3}$ as follows:
  \[X_{t+1}=\begin{cases}
   X_t, \quad&\text{w.p. } \mathbb{P}[Y_1 \in H_{X_t} \mid Y_0\in H_{X_t}] \\
   X_t-1, \quad&\text{w.p. } \mathbb{P}[Y_1 \in H_{X_t-1} \mid Y_0\in H_{X_t}]\\
   X_t+1, \quad&\text{w.p. } \mathbb{P}[Y_1 \in H_{X_t+1} \mid Y_0\in H_{X_t}] \\
  \end{cases}\]
  where $Y_1\in \Omega$ follows from $Y_0\in \Omega$ according to the Rook's Walk with parameters $n,d$ where $\Omega$ is the state space of said Rook's Walk.

\end{definition}
 \noindent From this definition, one can very simply describe the stationary distribution of the RWBD chain (denoted $\pi_*$) with respect to the stationary distribution of the standard Rook's walk (denoted $\pi$) using the Hamming shell counting function $\varphi$:
     \[\pi_*(x) = \varphi(x)\pi(\hat{1}) = \varphi(x)n^{-d},\quad\forall x\in \{0,\dots,d\}\]
  The formulation of the RWBD chain above is difficult to work with so we now find a concrete construction of the RWBD. Firstly, let $Y_1\in \Omega$ follow from $Y_0\in H_{X_t}$ based on the standard Rook's walk. Trivially, $Y_1$ is one of $d(n-1)$ states with equal probability. For $Y_1$ to stay within the same shell (i.e. $Y_1\in H_{X_t}$), a non-one coordinate needs to be changed to a non-one value. There are $X_t$ non-one coordinates and each coordinate can be changed to $n-2$ different values. Thus,
  \[\mathbb{P}[Y_1 \in H_{X_t} \mid Y\in H_{X_t}] = \frac{X_t(n-2)}{d(n-1)}\]
Moreover, for $Y_1$ to move down a shell (i.e $Y_1\in H_{X_t-1}$), a non-one coordinate needs to be changed to be one. There are exactly $X_t$ non-one coordinates so,
  \[\mathbb{P}[Y_1 \in H_{X_t-1} \mid Y_0\in H_{X_t}] = \frac{X_t}{d(n-1)}\]
  Finally, the Rook's walk only changes one coordinate each step so $Y_1\in H_{X_t}\cup H_{X_t-1} \cup H_{X_t+1}$. Thus,
  \begin{align*}
  \mathbb{P}[Y_1 \in H_{X_t+1} \mid Y_0\in H_{X_t}] &= 1 - \frac{X_t}{d(n-1)} - \frac{X_t(n-2)}{d(n-1)}\\
  &= 1 -  \frac{X_t}{d}\\
  &= \frac{d-X_t}{d}
  \end{align*}
  Hence the transition matrix $P^*$ for the RWBD chain with parameters $n,d$  is,
  \[P^*(x,y)=\begin{cases}
   \frac{y(n-2)}{d(n-1)}, \quad &x=y \\
   \frac{y}{d(n-1)}, \quad &x = y-1\\
   \frac{d-y}{d},\quad &x=y+1\\
   0 &\text{otherwise}
  \end{cases}\]for any states $x,y\in \Omega^*$. Using this formulation we can now prove the connection between the variation distances of both chains
\begin{lemma}
   \label{d}
  For any $n\in \mathbb{Z}_{\geq 3}$ and $d\in \mathbb{Z}_{ \geq 2}$, let   $\Omega,\Omega_*$, $P,P_*$ and $\pi,\pi_*$ be the state space, transition matrix and stationary distribution of the standard Rook's walk and the RWBD chains resp. each with parameters $n$ and $d$. Also, let $x\in \Omega_*$. Then,
  \[\Vert P_*^t(x,\cdot) - \pi_*(\cdot) \Vert_\text{TV} = \Vert (P^t \mu_x)(\cdot) - \pi(\cdot) \Vert_\text{TV}\]
  where $\mu_x$ is a distribution of $\Omega$ with probabilities uniformly distributed over $H_x$. That is,
   \[\mu_x(h) = \begin{cases}
      \frac{1}{\varphi(x)}, \quad &h\in H_x\\
      0, \quad &h\not\in H_x
   \end{cases}\]
   for all $h\in \Omega$, where $\varphi(x) = {d \choose i}(n-1)^i$ as in Definition \ref{defphi}.
   \begin{proof}
      \begin{align*}
         \Vert P_*^t(x,\cdot) - \pi_*(\cdot) \Vert_\text{TV} &= \frac{1}{2}\sum_{y\in \Omega_*}|P^t_*(x,y) -\pi_*(y)|\\
         &= \frac{1}{2}\sum_{y\in \Omega_*}|(P^t \mu_x)(H_y) -\pi(H_y)| \tag{By Def. \ref{BDRW}}\\
         &= \frac{1}{2}\sum_{y\in \Omega_*}\sum_{h_y\in H_y}|(P^t \mu_x)(h_y) -\pi(h_y)|\\
         &= \frac{1}{2}\sum_{z\in \Omega}|(P^t \mu_x)(z) -\pi(z)|\\
         &= \Vert (P^t \mu_x)(\cdot) - \pi(\cdot) \Vert_\text{TV}
      \end{align*}
   \end{proof}
\end{lemma}

\noindent We now prove that it suffices to consider only the starting state $0 \in \Omega_*$ for the RWBD chain.
\begin{lemma}
   \label{dd}
   Let $P_*$ and $\pi_*$ be the transition matrix and stationary distribution resp. for the RWBD chain with parameters $n,d$ and state space $\Omega_*$. Then the distance to stationary $d_*(t)$ at time $t$ is,
   \[\max_{x\in \Omega_*}\Vert P^t_*(x,\cdot) - \pi_*  \Vert_\text{TV} = \Vert P^t_*(0,\cdot) - \pi_*\Vert_\text{TV} \]
   \begin{proof}
      Fix $x\in \Omega_*$, and let $\delta_x$ be the distribution over $\Omega$ with all weight at $x$. Then,
      \begin{align*}
  \Vert P_*^t(x,\cdot) - \pi_*(\cdot) \Vert_\text{TV} &= \Vert (P^t \mu_x)(\cdot) - \pi(\cdot) \Vert_\text{TV}\tag{By Lemma \ref{d}}\\
   &= \left\Vert \sum_{x\in H_x}\frac{1}{\varphi(x)}(P^t\delta_x)(\cdot) - \pi(\cdot) \right\Vert_\text{TV}\\
   &= \left\Vert \sum_{x\in H_x}\frac{1}{\varphi(x)}\left((P^t\delta_x)(\cdot) - \pi(\cdot)\right) \right\Vert_\text{TV}\\
   &\leq \sum_{x\in H_x}\left\Vert\frac{1}{\varphi(x)} \left((P^t\delta_x)(\cdot) - \pi(\cdot)\right) \right\Vert_\text{TV}\tag{By Triangle Inequality}\\
   &= \sum_{i=1}^{\varphi(x)}\frac{1}{\varphi(x)}\left\Vert (P^t\delta_{\hat{1}})(\cdot) - \pi(\cdot) \right\Vert_\text{TV}\tag{By Transitivity}\\
   &= \left\Vert (P^t\delta_{\hat{1}})(\cdot) - \pi(\cdot) \right\Vert_\text{TV}\\
   &= \Vert P_*^t(0,\cdot) - \pi_*(\cdot) \Vert_\text{TV}
      \end{align*}
      
   \end{proof}
\end{lemma}


\noindent Finally, we can now prove the Rook's Walk has the same mixing time as the RWBD chain as an easy corollary.
\begin{corollary}
   \label{mixing_time_eq}
     Let $t_\text{mix}$ and $\hat{t}_{\text{mix}}$ be the mixing times of the standard Rook's walk and the RWBD resp. (each with parameters $n,d$). Then:
     \[\forall \epsilon \in (0,1), \quad t_\text{mix}(\epsilon) = \hat{t}_{\text{mix}}(\epsilon)\]
\begin{proof}
   Consider the variation distance of the RWBD chain (denoted $d_*(t)$),
   \begin{align*}
   d_*(t) &=  \max_{x\in \Omega_*}\Vert P^t_*(x,\cdot) - \pi_*  \Vert_\text{TV}\\
    &= \Vert P^t_*(0,\cdot) - \pi_*\Vert_\text{TV} \tag{Lemma \ref{dd}}\\
    &= \Vert P^t(\hat{1},\cdot) - \pi\Vert_\text{TV}\tag{Lemma \ref{d}}\\&=d(t)\tag{ Lemma \ref{transitivity}}
   \end{align*}
Thus, the distance from stationary at time $t$ of the two chains is equal and the equality of mixing times follows trivially.
\end{proof}
\end{corollary}

\subsection{Eigenvalues/functions of the Projection Chain}
In the previous section we defined a projection chain of the standard Rook's Walk called RWBD chain and proved that both chains have the same mixing time. This turns out to be very useful in analyzing the mixing time of the standard Rook's Walk chain as the RWBD is a type of birth-death chain with some extremely nice properties, especially with regard to spectral analysis. In particular, we are able to explicitly find the eigenfunctions and eigenvalues of the RWBD chain. We say a function $f$ over state space $\Omega$ is an eigenfunction of the transition matrix $P$ if there exists an eigenvalue $\lambda$ such that,
\[Pf = \lambda f\] 
Our method of determining the eigenfunctions of $P$ is to define an inner product over $\mathbb{R}^\Omega$ and find      a sequence of increasing degree polynomials that are pairwise orthogonal with respect to the inner product. This sequence of polynomials turns out to be the eigenfunctions of $P$. Hence, we now define an inner product.

\begin{definition}
  Let $\pi$ be a distribution over $\Omega$. Then we can define the inner product $\langle\cdot ,\cdot \rangle_\pi$ in $\mathbb{R}^{\Omega}$ as,
  \[\langle f, g\rangle_\pi = \sum_{x \in \Omega}f(x)g(x)\pi(x), \quad  \forall f,g \in \mathbb{R}^\Omega\]
\end{definition}
\noindent Before we describe the connection between polynomials orthogonal with respect to this inner product and eigenfunctions of $P$, we first state and prove that $P$ is self-adjoint with respect to our inner product. 
\begin{lemma}
   \label{self_ajoint}
   Let $P$ be the transition matrix of a reversible Markov chain; That is, $P$ satisfies,
   \[\pi(x) P(x,y) = \pi(y)P(y,x),\quad \forall x,y\in \Omega\]
   where $\pi$ is the stationary distribution. Then,
   \[\langle Pf,g\rangle_\pi = \langle f,Pg\rangle_\pi,\quad \forall f,g \in \mathbb{R}^\Omega\]
   \begin{proof}
     \begin{align*}
      \langle Pf, g\rangle_\pi &= \sum_{x\in \Omega}(Pf)(x)g(x)\pi(x)\\
      &= \sum_{x\in \Omega}\sum_{y\in \Omega}P(x,y)f(y)g(x)\pi(x)\\
      &= \sum_{x\in \Omega}\sum_{y\in \Omega}P(x,y)f(y)g(x)\pi(x)\\
      &= \sum_{y\in \Omega}\sum_{x\in \Omega}P(y,x)g(x)f(y)\pi(y)\tag{By reversibility}\\
      &= \sum_{y\in \Omega}(Pg)(y)f(y)\pi(y)\\
      &=  \langle f, Pg\rangle_\pi
     \end{align*} 
   \end{proof}
\end{lemma}
\noindent This lemma turns out to apply to any birth-death chain as all birth-death chains are reversible (\cite{levin_markov_2017} Prop. 2.8). Using this lemma we are able to prove the connection between orthogonal polynomials and eigenfunctions.
\begin{theorem}
   \label{eig-orth}
  Let $(X_t)_{t\in \mathbb{N}}$ be a birth-death chain over state space $\Omega=\{0,\dots,d\}$ with transition probabilities from some $x\in \Omega$,
  \begin{align*}
   P(x+1,x) &= p_x\\
   P(x-1,x) &= q_x\\
   P(x,x) &= r_x = 1 - p_x-q_x\\
   P(x,y) &= 0, \quad \forall y \not\in\{x,x+1,x-1\}
  \end{align*}
  Additionally, assume that $p_x$ and $q_x$ are affine in $x$. That is, there exists $\alpha,\beta,\xi,\eta\in \mathbb{R}$ such that,
  \[p_x = \alpha + \beta x \quad \text{and} \quad q_x = \xi + \eta x\]
  Finally, let $(K_m)_{m=0}^d$ be a sequence of degree $m$ polynomials that are orthogonal with respect to $\langle\cdot ,\cdot \rangle_\pi$ where $\pi$ is the stationary distribution. Then, $(K_m)_{m=0}^d$ are the eigenfunctions of $P$.
  \begin{proof}
   Let $x\in \Omega$ and $j\in \{0,\dots,d\}$ and consider,
   \begin{align*}
      (PK_j)(x) &= \sum_{y\in \Omega} P(x,y)K_j(y)\\
      &= P(x,x+1)K_j(x+1) + P(x,x-1)K_j(x-1) + P(x,x)K_j(x)\\
      &= q_{x+1} K_j(x+1) + p_{x-1} K_j(x-1) + r_x K_j(x)\\
      &= (\xi + \eta (x-1))K_j(x+1) + (\alpha + \beta (x+1))K_j(x-1) +(1 -(\alpha +\xi) - (\beta + \eta)x) K_j(x)
   \end{align*}
   Now, let $c$ be the coefficient of $x^{j}$ in $K_j(x)$. That is,
   \[K_j(x) = cx^j + o(x^{j-1})\]
   At this point, we examine the coefficient of $x^{j+1}$ in $(PK_j)(x)$. Using the above form and considering only relevant terms we have
   \begin{align*}
   (PK_j)(x)&= (\xi -\eta + \eta x)K_j(x+1) + (\alpha + \beta + \beta x)K_j(x-1) +(1 -(\alpha +\xi) - (\beta + \eta)x) K_j(x)\\
   &= \eta xc(x+1)^j + \beta x c(x-1)^j - (\beta + \eta)x cx^j + o(x^j)\\
   &=  \eta  cx^{j+1}+ \beta cx^{j+1} - (\beta + \eta)cx^{j+1} +  o(x^j)\\
   &= 0\cdot x^{j+1} + o(x^j)
   \end{align*}
   Thus, the coefficient of $x^{j+1}$ in $(PK_j)(x)$ is $0$. Also, notice that $(PK_j)(x)$ has no terms of order larger than $x^{j+1}$. This implies that $PK_j \in \mathbf{P}_j$ where $\mathbf{P}_j$ is the vector space of polynomials with degree at most $j$. This means that $\mathbf{P}_j$ is an invariant subspace with respect to $P$. Now, for any fixed $m\in \{0,\dots,d\}$, we know that $\{K_0,\dots K_{m-1}\}$ are orthogonal and therefore linearly independent. Thus, $\{K_0,\dots K_{m-1}\}$ span $\mathbf{P}_{m-1}$. Hence, any polynomial in $\mathbf{P}_j$ where $j < m$ is also orthogonal to $K_m$ by linearity of the inner product. Thus, $PK_j \in \mathbf{P}_j$, is orthogonal to $K_m$. That is,
   \[\langle K_m, PK_j \rangle_\pi = 0\]
   Hence, by Lemma \ref{self_ajoint},
   \[\langle PK_m, K_j\rangle_\pi=0 \quad \text{for}\quad j<m\]
   Finally, $PK_m \in \mathbf{P}_m$ so it admits the following presentation,
   \[PK_m = a_0 K_0 + a_1 K_1 + \cdots + a_m K_m\]
   where $a_0,\dots, a_m$ are constants. Hence,
   \begin{align*}
   0&= \langle PK_m, K_j\rangle_\pi\\
   &= \langle a_0 K_0 + a_1 K_1 + \cdots + a_m K_m, K_j\rangle_\pi\\
   &= a_0\langle K_0, K_j\rangle + \cdots + a_m\langle K_m, K_j\rangle_\pi \tag{by linearity}\\
   &= a_j\langle K_j, K_j\rangle_\pi \tag{by orthogonality}
   \end{align*}
   Therefore, $a_j = 0$ for all $j<m$. Thus,
   \[PK_m = a_m K_m\]
   and therefore, $K_m$ is an eigenfunction.
   \end{proof} 
\end{theorem}
\noindent Theorem \ref{eig-orth} gives us a way to find the eigenfunctions of the RWBD chain as the RWBD chain has affine transition probabilities. We simply need functions that are orthogonal to the inner product corresponding to the stationary distribution of the RWBD chain. As it turns out, there is a class of polynomials called the Krawtchouk polynomials that happen to satisfy that condition. Therefore, we state the definition of the Krawtchouk polynomials and a Lemma regarding their orthogonality from \cite{coleman_krawtchouk_2011}.
\begin{definition}
   \label{krav_def}
   Let $N,s\in \mathbb{N}$ such that $s\geq 2$. Then the {\bf Krawtchouk polynomials}, $(K_m)_{m\leq N}$ are,
   \[K_{m,N,s}(x) = K_m(x) = \sum_{j=0}^m (-1)^j{x \choose j}{N-x\choose m-j}(s-1)^{m-j}\]
   Note that $x$ is not required to be integer valued which conflicts with the traditional definition of the binomial coefficient. Therefore we use a more general definition of the binomial coefficient that uses falling factorials; That is,
   \[{x \choose 0} = 1, \quad {x \choose j} = \frac{1}{j!}(x)(x-1)\cdots(x-j+1) \quad \text{for} \quad j\geq 1\]
\end{definition}
\begin{lemma}
   \label{krav_orth}
   Let $N,s\in \mathbb{N}$ such that $s\geq 2$. Then, the Krawtchouk polynomials $(K_{m,N,s})_{m\leq N}$ are orthogonal with respect to the following inner product,
   \[\langle A,B \rangle := \sum_{i=0}^n {N \choose i}(s-1)^iA(i) B(i)\]
   where $A,B$ are polynomials of degree at most $N$. This lemma is proven in Proposition 3.1 in \cite{coleman_krawtchouk_2011}.
\end{lemma}
\noindent At this point, we have everything needed to get all eigenfunctions and eigenvalues of the RWBD chain.
\begin{theorem}
   The Krawtchouk polynomial $K_{m,d,n} = K_m$ is an eigenfunction with eigenvalue $\lambda_m = 1 - \frac{nm}{d(n-1)}$ of the RWBD chain's transition matrix $P$.
   \begin{proof}
      By Lemma \ref{krav_orth}, the Krawtchouk polynomials are orthogonal wrt the inner product,
      \[\langle A,B\rangle = \sum_{i=0}^n {d \choose i}(n-1)^i A(i) B(i)=\langle A,B \rangle_\pi\]
      The transition probabilities of the RWBD chain are affine so by Theorem \ref{eig-orth}, the Krawtchouk polynomials are eigenfunctions as required. Therefore, for each $K_m$ there exists $\lambda_m$ such that,
      \[PK_m = \lambda_m K_m\]
      Specializing to the $0^\text{th}$ coordinate we have,
      \[(PK_m)(0) = \lambda_mK_m(0)\]
      Now, consider $(PK_m)(0)$
      \begin{align*}
      (PK_m)(0) &= \sum_{y\in \Omega} P(0,y)K_m(y)\\
      &= P(0,0)K_m(0)+P(0,1)K_m(1)\\& = K_m(1)
      \end{align*}
      Therefore,
      \[\lambda_m = \frac{K_m(1)}{K_m(0)}\]
      Evaluating $K_m(0)$ and $K_m(1)$,
      \begin{align*}
            K_m(0) &= \sum_{j=0}^{m}{(-1)}^j{0 \choose j}{d-0 \choose m-j}{(n-1)}^{m-j}\\
            &= {d \choose m}{(n-1)}^{m}\\
            K_m(1) &= \sum_{j=0}^{m}{(-1)}^j{1 \choose j}{d-1 \choose m-j}{(n-1)}^{m-j}\\
            &= {d-1 \choose m}(n-1)^m - {d-1 \choose m-1} {(n-1)}^{m-1}
      \end{align*}
      Recall the standard binomial coefficient identities,
      \[\frac{d}{d-m}{d \choose m}={d-1 \choose m} \quad \text{and}\quad \frac{m}{d}{d\choose m} = {d-1 \choose m-1} \]
      Therefore,
      \begin{align*}
         \lambda_m &= \frac{\frac{d}{d-m}{d \choose m}(n-1)^m - \frac{m}{d}{d \choose m} (n-1)^{m-1}}{{d \choose m}(n-1)^{m}}\\
         &= \frac{d}{d-m} - \frac{m}{d(n-1)} \\
         &= 1 - \frac{mn}{d(n-1)}
      \end{align*}
   \end{proof}
\end{theorem}
\noindent As a corollary, we now give a concrete formulation for one of the eigenvalues and corresponding eigenfunctions that we will turn out to be useful.
      \begin{corollary}
         \label{Wilson-eig}
         Let $P$ be the transition matrix of the RWBD chain. Then, $\Phi$ is an eigenfunction of $P$ with eigenvalue $\alpha$ if
         \[\alpha= 1 - \frac{n}{d(n-1)}\quad \text{and}\quad \Phi(x) = 1 \frac{nx}{d(n-1)}\]
         \begin{proof}
           Clearly $\alpha = \lambda_1$. Hence, consider $K_1(x)$
            \begin{align*}
            K_1(x) &= \sum_{j=0}^1 (-1)^j{x \choose j}{d-x \choose 1-j}(n-1)^{1-j}\\
            &= (-1)^0{x\choose 0}{d-x \choose 1}(n-1)^{1} + (-1)^1{x \choose 1}{d-x\choose 0}(n-1)^{0}\\
            &= (d-x)(n-1) -x\\
            &= d(n-1) -nx\\
            &\propto 1 -\frac{nx}{d(n-1)} = \Phi(x)
            \end{align*}
            Thus, $\Phi(x)$ is an eigenfunction with eigenvalue $\alpha$.
         \end{proof}
      \end{corollary}

\section{RWBD Mixing Time Bounds}

\subsection{Lower Bound}
In the previous section we found the full spectrum of the RWBD chain and all corresponding eigenfunctions. This allows us to obtain a lower bound on the mixing time using Wilson's method (\cite{Wilson_2004}) as given in Theorem 13.28 from \cite{levin_markov_2017} stated below.
\begin{theorem}
   \label{wilson}
   Let $(X_t)$ be an ergodic Markov chain with state space $\Omega$ and transition matrix $P$. Let $\Phi$ be an eigenfunction of $P$ with real eigenvalue $\alpha$ satisfying $1/2 < \alpha< 1$. Let $R>0$ such that:
   \[\mathbb{E}_x\left(|\Phi(X_1)-\Phi(x)|^2\right)\leq R,\quad \forall x\in \Omega\]
   Then for any $x\in \Omega$ and $\epsilon\in (0,1)$,
   \[t_\text{mix}(\epsilon) \geq \frac{1}{2\log(1/\alpha)}\left[\log \left(\frac{(1-\alpha)\Phi(x)^2}{2R}\right) + \log\left(\frac{1-\epsilon}{\epsilon}\right)\right]\]
\end{theorem}
\noindent Naturally, we are able to use our eigenvalue and function found in Corollary \ref{Wilson-eig} to make use of Wilson's lower bound. However, we first need an upper bound on the variation of $\Phi$ from Corollary \ref{Wilson-eig}. Thus we have,
\begin{lemma}
   \label{wilson_var}
  Let $(X_t)_{t\in \mathbb{N}}$ be a RWBD chain with parameters $n,d$ and let $\Phi$ be defined as in Corollary \ref{Wilson-eig}. Then,
   \[\mathbb{E}_x\left(|\Phi(X_1)-\Phi(x)|^2\right)\leq \frac{n^2}{d^2(n-1)^2},\quad \forall x\in \Omega\]
   \begin{proof}
      Let $x\in \Omega$ be arbitrary and let $c = \frac{n}{d(n-1)}$. Then,
      \begin{align*}
         \mathbb{E}_x\left[|\Phi(X_1)-\Phi(x)|^2\right] &= \sum_{y\in \Omega} |\Phi(y) - \Phi(x)|^2\cdot P(y,x)\\
         &= |\Phi(x-1) - \Phi(x)|^2\frac{x}{d(n-1)} + |\Phi(x) - \Phi(x)|^2\frac{x(n-2)}{d(n-1)}  +|\Phi(x+1) - \Phi(x)|^2\frac{d-x}{d} \\
         &= |\Phi(x-1) - \Phi(x)|^2\frac{x}{d(n-1)}  +|\Phi(x+1) - \Phi(x)|^2\frac{d-x}{d} \\
         &= |1-c(x-1) - 1+cx|^2\frac{x}{d(n-1)}  +|1-c(x+1) - 1+cx|^2\frac{d-x}{d} \\
         &= c^2\frac{x}{d(n-1)}  +(-c)^2\frac{d-x}{d} \\
         &= \left(\frac{n}{d(n-1)}\right)^2\left(\frac{x}{d(n-1)}  +\frac{d-x}{d}\right) \\ \displaybreak[2]
         &= \left(\frac{n^2}{d^2(n-1)^2}\right)\left(\frac{x+(d-x)(n-1)}{d(n-1)}  \right) \\ \displaybreak[2]
         &= \left(\frac{n^2}{d^2(n-1)^2}\right)\left(\frac{x+d(n-1)-x(n-1)}{d(n-1)}  \right) \\ \displaybreak[2]
         &= \left(\frac{n^2}{d^2(n-1)^2}\right)\left(\frac{d(n-1)+x(2-n)}{d(n-1)}  \right) \\ \displaybreak[2]
         &= \left(\frac{n^2}{d^2(n-1)^2}\right)\left(1 + \frac{x(2-n)}{d(n-1)}  \right) \\
         &= \frac{n^2}{d^2(n-1)^2} + \frac{xn^2(2-n)}{d^3(n-1)^3} \\
         &= \frac{n^2(d(n-1)+x(2-n))}{d^3(n-1)^3} 
      \end{align*}
     Now, maximizing over $x\in \Omega$ we get that for all $x\in \Omega$,
     \begin{align*}
         \mathbb{E}_x\left[|\Phi(X_1)-\Phi(x)|^2\right] &\leq \frac{n^2(d(n-1)+0(2-n))}{d^3(n-1)^3} \tag{$\because n> 2$}\\
         &= \frac{n^2}{d^2(n-1)^2} 
     \end{align*}
   \end{proof}
\end{lemma}

\noindent Finally, we have what is needed to prove our lower bound. 
   \begin{proof}[Proof of Lower Bound in Theorem \ref{main1}]
      From Theorem \ref{wilson} we have,
   \[t_\text{mix}(\epsilon) \geq \frac{1}{2\log(1/\alpha)}\left[\log \left(\frac{(1-\alpha)\Phi(x)^2}{2R}\right) + \log\left(\frac{1-\epsilon}{\epsilon}\right)\right], \quad \forall x\in \Omega\]
   where $\alpha = 1-\frac{n}{d(n-1)}$, $R=\frac{n^2}{d^2(n-1)^2}$ and $\Phi(x) = 1-(1-\alpha)x$. Maximizing over $x\in \Omega$,
   \begin{align*}
   t_\text{mix}(\epsilon) &\geq \frac{1}{2\log(1/\alpha)}\left[\log \left(\frac{(1-\alpha)\Phi(0)^2}{2R}\right) + \log\left(\frac{1-\epsilon}{\epsilon}\right)\right]\\
   &= \frac{1}{2\log(1/\alpha)}\left[\log \left(\frac{1-\alpha}{2R}\right) + \log\left(\frac{1-\epsilon}{\epsilon}\right)\right]
   \end{align*}
   Consider just $\frac{1-\alpha}{2R}$,
   \begin{align*}
      \frac{1-\alpha}{2R}&= \frac{\frac{n}{d(n-1)}}{2\frac{n^2}{d^2(n-1)^2}}\\
      &= \frac{nd^2(n-1)^2}{2d(n-1)n^2}\\
      &= \frac{d(n-1)}{2n}
   \end{align*}
   Hence,
   \[t_\text{mix}(\epsilon) \geq \left(\log \left(\frac{d(n-1)}{2n}\right) + \log\left(\frac{1-\epsilon}{\epsilon}\right)\right)\left(2\log\left(\frac{d(n-1)}{d(n-1)-n}\right)\right)^{-1}\]
   as required.
   \end{proof}

\subsection{Upper Bound}
In Section 2 we found every eigenvalue and eigenfunction of the RWBD chain. Here we make use of them to obtain our upper bound using an $L^2$ bound. We now state the formulation of the $L^2$ bound from \cite{levin_markov_2017}.
\begin{lemma}[Lemma 12.18 from \cite{levin_markov_2017}]
   \label{l2}
   Let $P$ be the transition matrix for the RWBD chain with state space $\Omega = \{0,\dots,d\}$ and stationary distribution $\pi$. Further, let $\phi_m$ and $\lambda_m$ be orthonormal eigenfunctions and eigenvalues of $P$. Then,
   \[4\Vert P^t(x,\cdot) - \pi \Vert_\text{TV}^2 \leq \left\Vert\frac{P^t(x,\cdot)}{\pi(\cdot)}-1\right\Vert_2^2 = \sum_{m=1}^d \phi_m^2(x) \lambda_m^{2t}\]
   for any $x\in \Omega$ and $t\geq 0$.
\end{lemma}
\noindent Clearly, our set of orthogonal eigenfunctions can be put to use given this lemma. However, we need to normalize our orthogonal eigenfunctions. Thus, for every $m\in \Omega$ let,
\[\phi_m(x) = \frac{K_m(x)}{\langle K_m, K_m \rangle_\pi}\]
Now, we present a useful fact about $\phi_m$,
\begin{lemma}
   \label{choose}
   Fix $m\in \{0,\dots,d\}$ and let $\phi_m$ be as above. Then,
   \[\phi_m^2(0) \leq {d \choose m}\]
   \begin{proof}
     
     \[\phi_m^2(x) = \frac{K_m^2(x)}{\langle K_m, K_m\rangle_\pi}\]
     Specializing to $x=0$, by Def. \ref{krav_def} we have
   \begin{align*}
   K_m(0) &= \sum_{j=0}^m (-1)^j(0)_j(d-0)_{m-j}(n-1)^{m-j}\\
   &= (d)_{m}(n-1)^{m} \tag{$\because (0)_j = 0, \quad\forall j>0$}\\
   \therefore \quad K_m^2(0) &= (d)^2_{m}(n-1)^{2m}
   \end{align*}
     Proposition 3.1 in \cite{coleman_krawtchouk_2011} gives,
     \[\langle K_m, K_m\rangle_\pi = n^d(n-1)^m(d)_m\]
   Therefore,
   \[\phi^2_m(0) = \frac{(d)_m^2(n-1)^{2m}}{n^d(n-1)^m(d)_m} = \frac{(d)_m(n-1)^{m}}{n^d} \leq (d)_m = {d \choose m}\]

   \end{proof}

\end{lemma}

   \begin{proof}[Proof of Upper Bound in Theorem \ref{main1}]
     By Lemmas \ref{l2} and \ref{choose} we have,
     \[4\lVert P^t(0,\cdot) - \pi\rVert_\text{TV}^2  \leq \sum_{m=1}^d {d \choose m} \lambda_m^{2t} \] 
     Now, using a first order Taylor series bound we have,
     \[\lambda_m^{2t} = \left(1 - \frac{mn}{d(n-1)}\right)^{2t} \leq \exp\left(-\frac{mn}{d(n-1)}\right)^{2t}\].
     Let $\alpha = \frac{2tn}{d(n-1)}$ so that $\lambda_m^{2t} \leq e^{-\alpha m}$. Thus,
     \begin{align*}
     4\lVert P^t(0,\cdot) - \pi\rVert_\text{TV}^2  &\leq \sum_{m=1}^d {d \choose m}e^{-\alpha m}\\
     &= -{d \choose 0}e^{0}+ \sum_{m=0}^d {d \choose m}e^{-\alpha m}\\
     &= \sum_{m=0}^d {d \choose m}\left(e^{-\alpha}\right)^m (1)^{d-m}\\
     &=(1+e^{-\alpha})^d -1
     \end{align*}
     Therefore by Lemma \ref{dd},
     \[d(t) \leq \frac{1}{2}\sqrt{(1+e^{-\alpha})^d-1}\]
     Now, by the definition of mixing time, if the RHS above is equal to $\epsilon$ at some time $t$, then $t_\text{mix}(\epsilon) \leq t$. Thus consider,
     \begin{align*}
      \epsilon &= \frac{1}{2} \sqrt{(1+e^{-\alpha})^d-1}\\
      4\epsilon^2 +1&= (1+e^{-\alpha})^d\\
      (4\epsilon^2 +1)^{1/d}-1&= e^{-\alpha}\\
      \log\left((4\epsilon^2 +1)^{1/d}-1\right)&= -\frac{2tn}{d(n-1)}\\
      -\frac{d(n-1)}{2n}\log\left((4\epsilon^2 +1)^{1/d}-1\right)&= t
     \end{align*}
     Therefore,
   \[t_\text{mix}(\epsilon) \leq -\frac{d(n-1)}{2n}\log\left((4\epsilon^2+1)^{1/d}-1\right)\]
   as required.
   \end{proof}

  \section{Cutoff}
  In the previous section we proved Theorem \ref{main1}. That is, we found tighter upper and lower bounds on the mixing time of the RWBD, and by extension the Rook's walk Markov chains. Using these bounds, we are able to prove Theorem 2, namely, that the sequence of Rook's Walk (and RWBD) Markov chains indexed by $d$ for some fixed $n$, has cutoff at time $\frac{d(n-1)}{2n}\log d$ with window $\frac{d(n-1)}{n}$. The proof of this is split into the following two Lemmas.
  \begin{lemma}
   Let $n,d$ be any parameters for the Rook's Walk. Then for $\epsilon\in(0,1)$,
   \[t_\text{mix}(\epsilon) \geq \frac{d(n-1)}{2n} \log d + c_l(\epsilon) \frac{d(n-1)}{n}\]
   where $c_l(\epsilon)$ is a constant dependent only on $\epsilon$.
   \begin{proof}
     Firstly, consider the lower bound from Theorem \ref{main1},
      \[t_\text{mix}(\epsilon) \geq \underbrace{\left(\log \left(\frac{d(n-1)}{2n}\right) + \log\left(\frac{1-\epsilon}{\epsilon}\right)\right)}_\mathbf{A} \underbrace{\left(2\log\left(\frac{d(n-1)}{d(n-1)-n}\right)\right)^{-1}}_\mathbf{B}\]
      We proceed by bounding $\mathbf{A}$ and $\mathbf{B}$ separately.
      \begin{align*}
         \mathbf{A} &:=  \log \left(\frac{d(n-1)}{2n}\right) + \log\left(\frac{1-\epsilon}{\epsilon}\right)\\
         &=  \log(d/2) + \log\left(\frac{n-1}{n}\right)+\log\left(\frac{1-\epsilon}{\epsilon}\right)\\
         &\geq  \log(d/2) + \log\left(\frac{1-\epsilon}{\epsilon}\right)+1
      \end{align*}
      \begin{align*}
         \mathbf{B} &:= \left(2\log\left(\frac{d(n-1)}{d(n-1)-n}\right)\right)^{-1}\\
         \frac{1}{2\mathbf{B}} &= \log\left(\frac{d(n-1)}{d(n-1)-n}\right)\\
         \frac{1}{2\mathbf{B}} &= -\log\left(\frac{d(n-1)-n}{d(n-1)}\right)\\
         \frac{1}{2\mathbf{B}} &= -\log\left(1 - \xi\right)\tag{Using $\xi:=\frac{n}{d(n-1)}$}
      \end{align*}
      Now, recall the Taylor expansions of $-\log(1-\xi)$ and $\frac{\xi}{1-\xi}$ around zero for $\xi\in (0,1)$,
      \[-\log(1-\xi) = \sum_{k=1}^\infty \frac{\xi^k}{k} \quad \text{and} \quad \frac{\xi}{1-\xi} = \sum_{k=1}^\infty \xi^k \]
      Therefore, $-\log(1-\xi) \leq \frac{\xi}{1-\xi}$. Thus,
      \begin{align*}
         \frac{1}{2\mathbf{B}} &\leq \frac{\xi}{1-\xi}\\
         \mathbf{B} &\geq \frac{1-\xi}{2\xi}\\
         \mathbf{B} &\geq \frac{1}{2\xi} - \frac{1}{2}\\
         \mathbf{B} &\geq \frac{d(n-1)}{2n} - \frac{1}{2}
      \end{align*}
      Combining our bounds for $\mathbf{A}$ and $\mathbf{B}$ and absorbing lower order terms we have,
      \begin{align*}
         t_\text{mix}(\epsilon) &\geq \left(\log(d/2) + \log\left(\frac{1-\epsilon}{\epsilon}\right)+1\right)\left(\frac{d(n-1)}{2n} - \frac{1}{2}\right)\\
         &\geq \frac{d(n-1)}{2n}\log(d/2) + \log\left(\frac{1-\epsilon}{\epsilon}\right)\frac{d(n-1)}{2n}- \frac{1}{2}\log(d/2) -\frac{1}{2}\log\left(\frac{1-\epsilon}{\epsilon}\right)\\
         &\geq \frac{d(n-1)}{2n}\log(d/2) + \left(\log\left(\frac{1-\epsilon}{\epsilon}\right)-2\right)\frac{d(n-1)}{2n} -\frac{1}{2}\log\left(\frac{1-\epsilon}{\epsilon}\right)\\
         &\geq \frac{d(n-1)}{2n}\log(d/2) + \left(\log\left(\frac{1-\epsilon}{\epsilon}\right)-6\right)\frac{d(n-1)}{2n}\\
         &\geq \frac{d(n-1)}{2n}\log(d/2) + c_l(\epsilon)\frac{d(n-1)}{n}
      \end{align*}
      as required.
   \end{proof}
  \end{lemma}

  \begin{lemma}
   Let $n,d$ be any parameters for the Rook's Walk. Then for $\epsilon\in(0,1)$,
   \[t_\text{mix}(\epsilon) \leq \frac{d(n-1)}{2n} \log d + c_u(\epsilon) \frac{d(n-1)}{n}\]
   where $c_u(\epsilon)$ is a constant dependent only on $\epsilon$.
   \begin{proof}
      Let $u=4\epsilon^2+1$. Then,
      \[u^{1/d} =e^{\log(u^{1/d})}=e^{\frac{1}{d}\log(u)}\]
      Now, $e^x$ is concave up so the first order Taylor approximation is a lower bound; That is, $e^x \geq 1+x$. Therefore,
      \[e^{\frac{1}{d}\log (u)} \geq 1+ \frac{1}{d}\log( u)\]
      or equivalently,
      \[u^{1/d} -1 \geq \frac{\log u}{d}\]
      Now, recall our upper bound from Theorem \ref{main1} and substitute,
      \begin{align*}
          t_\text{mix}(\epsilon) &\leq -\frac{d(n-1)}{2n}\log\left((4\epsilon^2+1)^{1/d}-1 \right)\\
          &= \frac{d(n-1)}{2n}\log\left(\frac{1}{(4\epsilon^2+1)^{1/d}-1 }\right)\\
          &\leq \frac{d(n-1)}{2n}\log\left(\frac{d}{\log (4\epsilon^2 +1)}\right)\\
          &= \frac{d(n-1)}{2n}\left(\log(d) -\log\log(4\epsilon^2+1)\right)\\
          &= \frac{d(n-1)}{2n}\log(d) +c_u(\epsilon) \frac{d(n-1)}{n}
      \end{align*}
      as required.
      
   \end{proof}
  \end{lemma}
\section{Concluding remarks}\label{sec:conclusion}
Here we sharpened the known total variation distance mixing-time bounds on the Rook's Walk with state space $\Omega=\{1,\dots,n\}^d$. As a consequence, we showed that the Rook's Walk has an asymptotic cutoff at time $\frac{d(n-1)}{2n}\log d$ with window $\frac{d(n-1)}{n}$. We achieved this by reducing the chain to a one-dimensional birth-death chain via a projection onto Hamming shells, computing the full spectral data of the projection, and combining Wilson's eigenfunction lower bound with an \(L^2\) upper bound.

A natural direction is to sharpen constants in the window and to understand whether the window order \(\frac{d(n-1)}{n}\) is optimal, or whether finer analysis of the leading eigenmodes yields a smaller cutoff window or a more descriptive cutoff profile describing the shape of the cutoff window. It would also be interesting to analyze other asymptotic regimes, for instance joint limits in which \(n=n(d)\) grows with \(d\). More broadly, the method here suggests studying other highly symmetric product (or near product) Markov chains where lumping produces a birth-death chain with an explicit orthogonal-polynomial diagonalization. Identifying when such projections preserve mixing times (or more generally how such lumpings change mixing times) may lead to further explicit cutoff results for related chains.

   \pagebreak
   \printbibliography 

\end{document}